\documentclass{birkjour_t2}
 \newtheorem{thm}{Theorem}[section]
 
 \newtheorem{lem}[thm]{Lemma}
 \newtheorem{prop}[thm]{Proposition}
 \theoremstyle{definition}
 
 \theoremstyle{remark}

 \numberwithin{equation}{section}

\begin{document}

%
%
%
%
%
%
%
%
%

\title[Lower bounds]{Some lower bounds for the maximal number of  $A$-singularities in algebraic surfaces}

\author[J.G.Escudero]{Juan Garc\'{\i}a Escudero}
\address{%
Madrid, Spain
}
\email{jjgemplubg@gmail.com}


\date{April,  2024}
\dedicatory{}

\begin{abstract}
We construct  algebraic surfaces with a large number of type $A$ singularities. Bivariate polynomials presented in previous works for the construction of nodal surfaces and certain families of Belyi polynomials are used. In some cases explicit expressions in terms of classical Jacobi polynomials are obtained.

\end{abstract}
\maketitle
\section{Introduction}
\bigskip\par
In \cite{esc13} a family of degree $3m$ polynomials was presented with the aim of improving the lower bounds for the maximal number of nodes or $A_{1}$ singularities in degree $3m$  complex algebraic surfaces. Besides the classical Cayley cubic \cite{ cay69} and Kummer quartic \cite{kum64} surfaces, previous lower bounds were obtained in  \cite{ tog40, chm92, bar96,end97, sar01,lab06b} and  upper bounds in \cite{ bea82, var83, miy84, jaf97} (some images of low degree surfaces with many nodes can be seen in \cite{bae16}). 
 \par
We denote by $\mu_{A_{\nu}}(d)$ ($\mu^{(\Bbb{R})}_{A_{\nu}}(d)$) the maximal number of $A_{\nu}$ complex (real) singularities for a degree $d$ surface in the complex (real) projective space  $\mathbb{P}^{3}({\Bbb{C}})$ ($\mathbb{P}^{3}({\Bbb{R}})$). General lower bounds for the number of nodes on degree $d$ complex algebraic surfaces were given in \cite{chm92}. The affine equations describing the surfaces consist of the sum of the classical univariate Chebyshev polynomials and the bivariate folding polynomials studied in \cite{wit88, hof88}. By considering a different class of bivariate polynomials, which were motivated by previous studies on substitution tilings constructed with the help of certain deltoid tangents, new lower bounds for nodes were obtained in \cite{esc13}. In \cite{esc16} we showed that there are real surfaces with the same number of real singularities as those studied in \cite{esc13}. The surfaces can also be defined over the rationals (or over the integers, after clearing the denominators) as shown in \cite{esc18} and the lower bounds are
\begin{equation} 
\mu^{(\Bbb{R})}_{A_{1}}(3m)\ge \frac{3m (3m-1)}{ 2}\lfloor\frac{3m}{2}\rfloor+\left(3m(m-1)+1\right)\lfloor \frac{3m-1}{2}\rfloor
\end{equation}
with explicit equations for the surfaces.
 \par
With the purpose of getting surfaces with many  $A_{\nu}$ singularities with $\nu>1$, a certain class of Belyi polynomials (used instead of Chebyshev polynomials) together with the folding polynomials were studied in \cite{lab06}. By using other types of Belyi polynomials in addition to the bivariate polynomials presented in \cite{esc13}, hypersurfaces with many non-nodal singularities were also constructed in \cite{esc14a, esc14b}. For cusps or $A_{2}$ singularities, we studied in \cite{esc14a} several special cases improving existing lower bounds and the results were extended in (\cite{esc14b}, Prop. 3.4), where we obtained: 
\begin{equation}
\mu_{A_{2}}(3m)\ge \frac{3m^{2}(3m-1)}{2}+(3m(m-1)+1)\lfloor \frac{m-1}{2}\rfloor
\end{equation}
An explicit equation for $d=9$ with 127 complex cusps is given in (\cite{esc14a}, Prop. 2).
 \par
In the cases of singularities of type $A_{\nu}, \nu>2$ some results for low degrees studied in \cite{esc14a} were also extended in (\cite{esc14b}, Prop. 3.6) giving:
\begin{equation} 
\mu_{A_{3m+1}}(3(2m+1))\ge 3m(10m+7)+4
\end{equation}
In particular we showed that there are explicit equations for real surfaces of degree 9 with 55 real  $A_{4}$ singularities (\cite{esc14a}, Fig. 2(c)) and degree 15 with 166 real  $A_{7}$ singularities (\cite{esc14b}, Fig. 3).
 \par
In this work we generalise the results given in \cite{esc14a, esc14b}. In order to investigate the existence of surfaces with many singularities, we use the family of bivariate polynomials  ${\mathcal{J}}$ defined over ${\Bbb{Q}}$ considered in \cite{esc18} in conjunction with several families of Belyi polynomials, which we study in Section 2.  In Section 3 we first recall the results about nodal hypersurfaces obtained with ${\mathcal{J}}$ in \cite{esc18}, and then we construct families of surfaces with many $A_\nu$ singularities, improving the corresponding lower bounds given in  \cite{lab06}. We also get explicit expressions in terms of Jacobi polynomials for some cases in Section 4.
 
 \section{Belyi polynomials}
\bigskip\par
If $w_{0}$ is a zero of a polynomial $P(w), w \in {\Bbb{C}}$, with critical value $\zeta=P(w_{0})$, then the order of a zero $w_{0}$ of $\frac{dP(w)}{dw}$ is called its multipicity $\nu$ (all the derivatives of $P(w)$ up to order $\nu$ vanish at $w_{0}$). A univariate polynomial with no more than two different critical values is called a Belyi polynomial. A graph without cycles (plane tree) with a bicoloring for the vertices is used to represent a Belyi polynomial whose critical points have the multiplicities given by the number of edges adjacent to the vertices minus one \cite{adr98, lab06}. The degree of a vertex is the number of edges incident to it. A leaf vertex is a vertex with degree one. Black and white vertices represent critical points with critical value $\zeta=-1$ and $\zeta=1$ respectively. Well known examples of Belyi polynomials are the degree $d$ Chebyshev polynomials of the first kind $T_{d}(w)$. The tree corresponding to $T_{9}(w)$ is shown in Fig.1, where the two leaf vertices correspond to non-critical points and the eight degree 2 vertices represent critical points with multiplicity 1.
 \par
We use the notation  ${\mathcal{B}}_{d,\nu, \epsilon}(w)$ for a degree $d$ Belyi  polynomial having  critical points $w_{1}, w_{2},... w_{n}$ of multiplicity $\nu$ with critical value $\zeta=-1$, one critical point $w_{0}$ of multiplicity $\nu$ with $\zeta=1$, and one additional critical point $w_{u}$ with multiplicity $\epsilon$ and $\zeta=-1$. The polynomials are then solutions of 

$$\frac{d {\mathcal{B}}_{d,\nu,\epsilon}(w)}{dw}=(w-w_{0})^{\nu}(w-w_{u})^{\epsilon}\prod_{l=1}^{n}(w-w_{l})^{\nu}$$
 \par\noindent
with ${\mathcal{B}}_{d,\nu,\epsilon}(w_{l})=-1, l=1,2,...,n; {\mathcal{B}}_{d,\nu,\epsilon}(w_{0})=1$ and  ${\mathcal{B}}_{d,\nu,\epsilon}(w_{u})=-1$ (the computation of some Belyi polynomials has been done in \cite{lab06, esc14a} by choosing one critical point equal to zero). 
 \par
When there is no critical point with multiplicity $\epsilon$ and $\zeta=-1$ we denote the polynomial  either by  ${\mathcal{B}}_{d,\nu,0}(w)$ or ${\mathcal{B}}_{d,\nu}(w)$. We also write ${\mathcal{B}}[k_{1},k_{2}...k_{l}]$ if the polynomial can be characterised by the parameters $k_{1},k_{2}...k_{l}$ as  in this Section final remarks.

         \begin{figure}[h]
 \includegraphics[width=20pc]{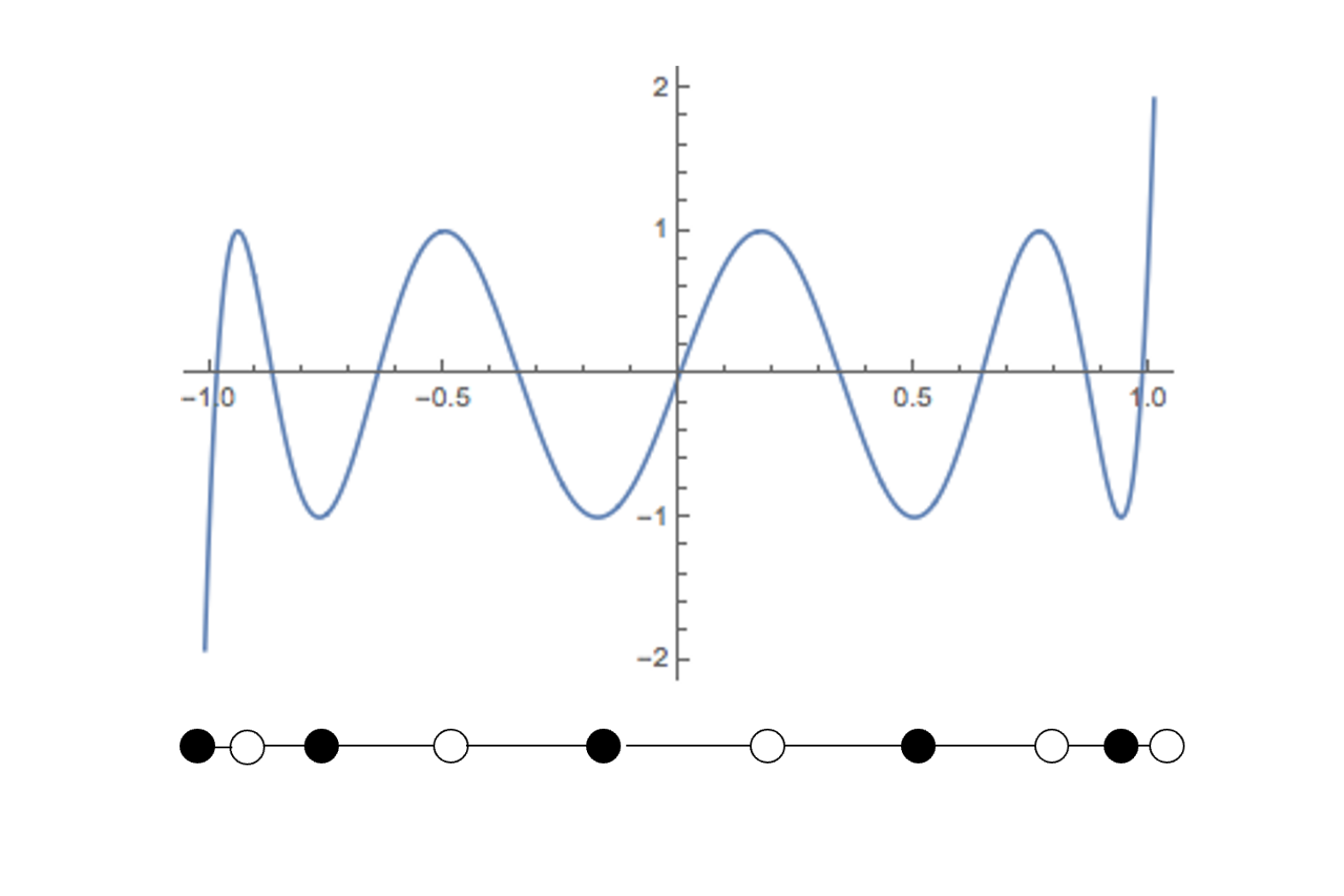}
\caption{\label{label} Chebyshev polynomial $T_{9}(w)$ (top) and its associated plane tree (bottom).}
\end{figure}   

 \begin{lem} 
There exist polynomials ${\mathcal{B}}^{(1)}_{d,\nu}(w)$, also denoted by ${\mathcal{B}}^{(1)}[n,m]$, with $k-1$ critical points of multiplicity $\nu$ with critical value $\zeta=-1$ and one critical point of multiplicity $\nu$ with critical value $\zeta=1$, where $k=3m+1, m\in {\Bbb{Z}}^{+}$ and
\begin{equation} 
d=k\nu  +1, \nu=3n+k-2, n \in {\Bbb{Z}}^{\geq 0}
\end{equation}
 \end{lem}
  \begin{proof}

  \par

The starting point, for $m=1$, is a plane tree made up of a central white vertex, 3 black vertices  and additional white vertices connected with the black ones (see Fig. 2(a) (left)). This corresponds to $ k=4, n=0$ in Eq. (2.1).  We then successively add $3n\in {\Bbb{Z}}^{+}$ edges (see the case $n=1$  in Fig. 2(a) (right)) connected to vertices with opposite colour (in Fig. 2(a) (right) we only show one edge with a number indicating the number of additional adjacent edges and omit 3 black and 9 white vertices). We get the trees associated to the series ${\mathcal{B}}^{(1)}_{12n+9,3n+2}(w), n \in {\Bbb{Z}}^{\geq 0}$.
  \par
  In order to construct the initial trees for each of the remaining series, which correspond to $n=0$ in Eq. (2.1), we add $k-1=3m, m=2,3,4,...$ edges to the central white vertex in Fig. 2(a) (left). In this way we obtain the trees for ${\mathcal{B}}^{(1)}_{(k-1)^2,k-2}(w)$. The case  $m=2$ is shown in Fig. 2(b) (left). 
   \par 
  We use the trees of ${\mathcal{B}}^{(1)}_{(k-1)^2,k-2}(w)$ as initial trees of the series obtained in the following way: for each value of $k$ we add $3n$ edges to each vertex in the corresponding tree, as in the $m=1$ case, and we get the series 
 
$${\mathcal{B}}^{(1)}_{(3(n+m)-1)(3m+1) +1,3(n+m)-1}(w), n\in {\Bbb{Z}}^{\geq 0},  m\in {\Bbb{Z}}^{+}$$
  In Fig. 2(b) (right) we represent  the next step  ($k=7$), which corresponds to the series ${\mathcal{B}}^{(1)}_{21n+36,3n+5}(w)$, $n\in {\Bbb{Z}}^{+}$.

         \end{proof}
         
         \begin{figure}[h]
 \includegraphics[width=30pc]{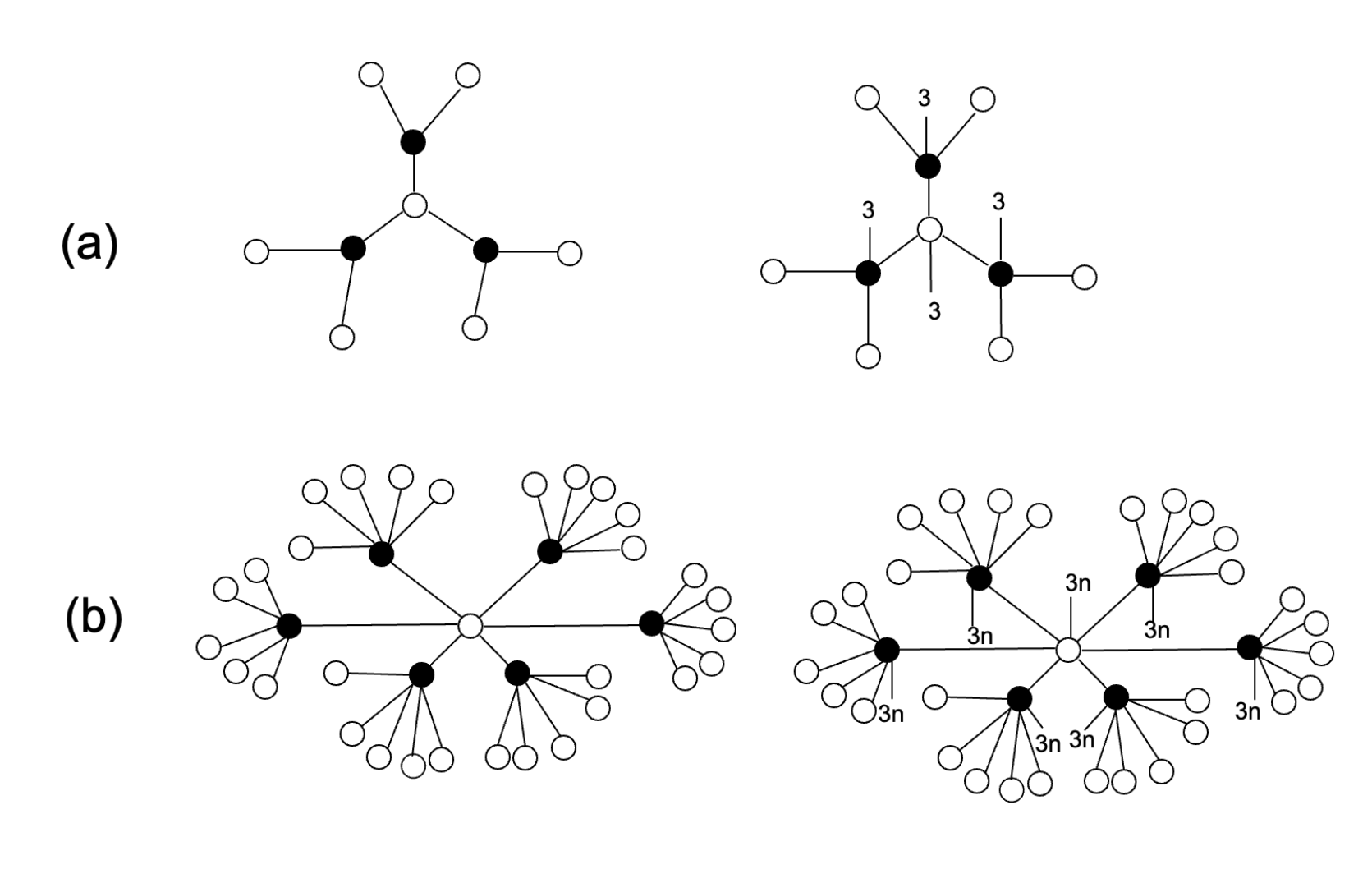}
\caption{\label{label}Plane trees for ${\mathcal{B}}^{(1)}_{d,\nu}(w)$: (a) $d=12n+9, \nu=3n+2$, $n=0$ (left),  $n=1$ (right), (b) $d=21n+36, \nu=3n+5$, $n=0$ (left), $n>0$ (right)}
\end{figure}   
         
    \par
     \begin{lem} 
There exist polynomials ${\mathcal{B}}^{(2)}_{d,\nu, \epsilon}(w)$ $({\mathcal{B}}^{(2)}[j,n,m,l])$  with $b-1$ critical points of multiplicity $\nu$ with critical value $\zeta=-1$, one critical point of multiplicity $\nu$ with $\zeta=1$ and one additional critical point of multiplicity $\epsilon$ with $\zeta=-1$ where
\begin{equation} 
d=b\nu+a, \nu=3n+b-1,\epsilon=a-1, n\in {\Bbb{Z}}^{\geq 0}
\end{equation}
with $a= 3m+j, b=3l+j+1$, $j\in\{0,1\}$, $m\in {\Bbb{Z}}^{+}$ if $j=0$, $m\in {\Bbb{Z}}^{\geq 0}$ if $j=1$, and $l=m, m+1, m+2,...$. 
 \end{lem}
  \begin{proof}

  \par
The initial trees are generated as follows.
 \bigskip\par \noindent 
  (1) $j=0$.  For  $m\in {\Bbb{Z}}^{+}$ and $l=m, m+1, m+2,...$ the polynomials have trees made of a central white vertex with $b-1=3l$ edges connected to black vertices, each one connected to another $3l$ white vertices. In addition the central white vertex is connected to an additional black vertex $u$ which has $a-1=3m-1$ edges connected to white vertices. The corresponding Belyi polynomials are ${\mathcal{B}}^{(2)}_{3m+3l(3l+1), 3l, 3m-1}(w)$. The tree for $m=1, l=1$ is represented in Fig. 3(a) (left). 
     \par \noindent
  (2) $j=1$. For $m\in {\Bbb{Z}}^{\geq 0}, l=m, m+1, m+2,...$ there is a tree made up of a central white vertex with $b-1=3l+1$ edges connected to black vertices, and each black vertex is connected to $b$ white vertices. The central white vertex is also connected to an additional black vertex $u$ which has $a-1=3m$ edges connected to additional white vertices. The corresponding Belyi polynomials are
  ${\mathcal{B}}^{(2)}_{3m+1+(3l+1)(3l+2), 3l+1, 3m}(w)$. The tree for $m=0, l=1$ is represented in Fig. 3(b) (left). 
 \bigskip\par 
  Therefore the initial trees are described by 
  \begin{equation} 
d_{0}=b\nu_{0}+a, \nu_{0}=b-1,\epsilon_{0}=a-1 
\end{equation}
where $a= 3m+j, b=3l+j+1$, $j=0,1$, $m\in {\Bbb{Z}}^{+}$ if $j=0$, $m\in {\Bbb{Z}}^{\geq 0}$ if $j=1$, $l=m, m+1, m+2,...$. 
        \par
The trees in (1) and (2) are then used as initial trees for the following series.  For each tree in (1) and (2) we construct a series of trees by adding $3n$ edges connected to vertices of opposite color for all the non leaf vertices with the exception of $u$, which is the only black vertex that remains unchanged in the whole process. The two series we obtain in this way are:
 \bigskip\par \noindent 
  (3)  $j=0$, $d=(3n+3l)(3l+1)+3m,  \nu=3n+3l, \epsilon=\epsilon_{0}=2, n\in {\Bbb{Z}}^{\geq 0}$. In Fig. 3(a) (right) we represent the case $m=1, l=1$: $d=12n+15, \nu=3n+3, n\in {\Bbb{Z}}^{+}$.
        \par \noindent
  (4)  $j=1$, $d=(3n+3l+1)(3l+2)+3m+1,  \nu=3n+3l+1, \epsilon=\epsilon_{0}=0, n\in {\Bbb{Z}}^{\geq 0}$. For $m=0, l=0$ we have the series corresponding to the bounds of Eq. (1.3). In Fig. 3(b) (right) we see the trees for the series ${\mathcal{B}}^{(2)}_{15n+21,3n+4,0}(w), n\in {\Bbb{Z}}^{+}$  for $m=0, l=1$.
          
   \end{proof}

   \begin{figure}[h]
 \includegraphics[width=30pc]{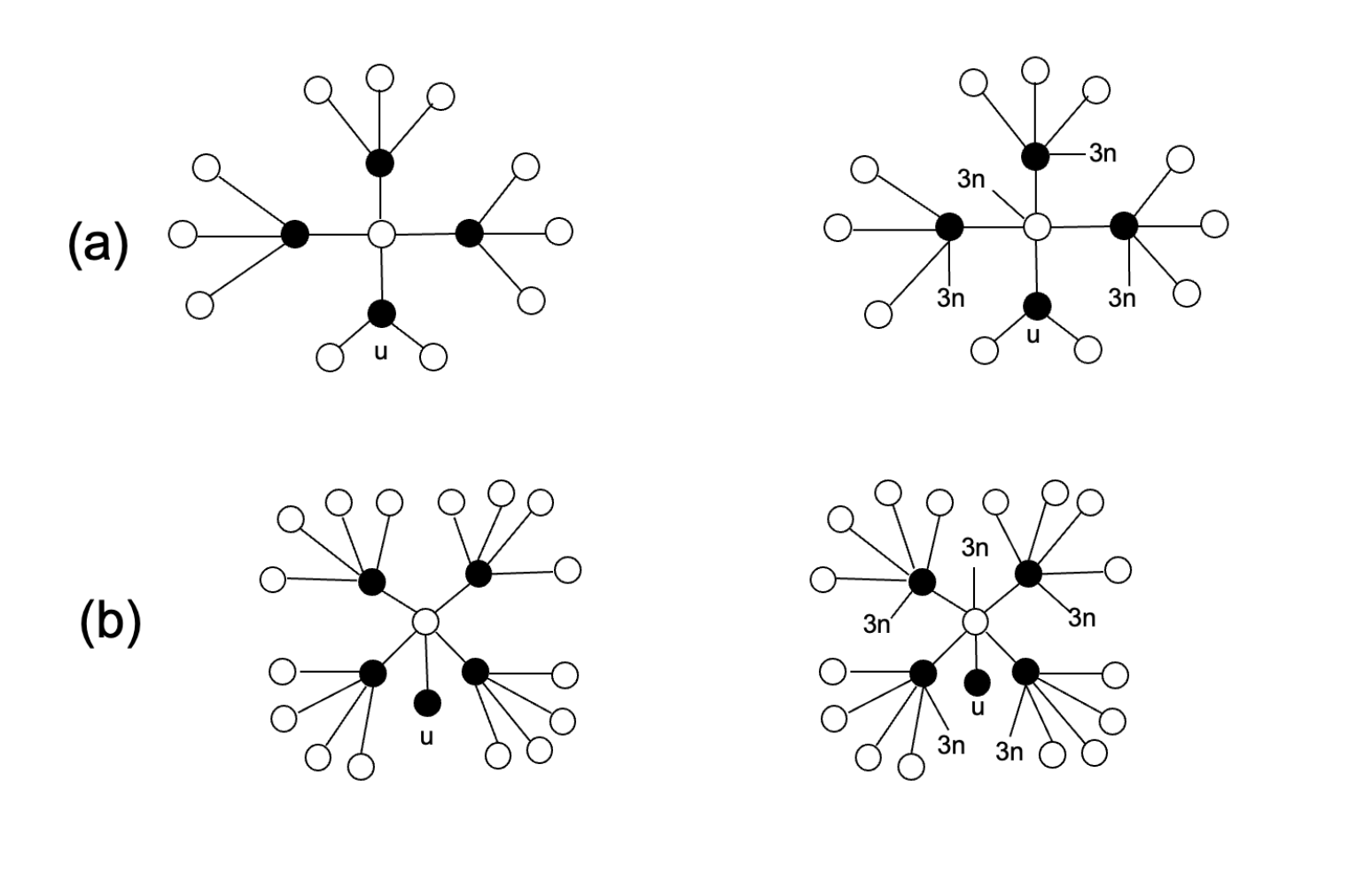}
\caption{\label{label}Plane trees for ${\mathcal{B}}^{(2)}_{d,\nu, \epsilon}(w)$: (a) $d=12n+15, \nu=3n+3, \epsilon=2$, $n=0$  (left)  and $n>0$ (right), (b)   $d=15n+21, \nu=3n+4, \epsilon=0$, $n=0$  (left)  and $n>0$ (right)}
\end{figure}

    \begin{lem} There exist polynomials ${\mathcal{B}}^{(3)}_{d,\nu, \epsilon}(w)$ $({\mathcal{B}}^{(3)}[x,p,n,m,l])$ with $p-1$ critical points of multiplicity $\nu=x+3n+p-1$ with critical value $\zeta=-1$, one critical point of multiplicity $\nu=x+3n+p-1$ with critical value $\zeta=1$, and one additional critical point with multiplicity $\epsilon=z$ and $\zeta=-1$ in the following cases 
 \bigskip\par \noindent
     (a) $x=1$: 
  \par
  (a1) $p=3m-1, d=3(np+m(3m-2)+l+1), \nu=3(n+m)-1, z=3l+1$
  \par
    (a2) $p=3m, d=3(np+3m^2+l+1), \nu=3(n+m),  z=3l+2$ 
  \par
   (a3)  $p=3m+1, d=3(np+m(3m+2)+l+1), \nu=3(n+m)+1, z=3l+1$ 
\par\noindent
  (b) $x=2$:
  \par
  (b1) $p=3m-1, d=3(np+m(3m-1)+l+1), \nu=3(n+m), z=3l+2$ 
\par
 (b2)  $p=3m, d=3(np+m(3m+1)+l+1), \nu=3(n+m)+1,  z=3l+2$  
\par
 (b3) $p=3m+1, d=3(np+3m(m+1)+l+1), \nu=3(n+m)+2,  z=3l$
\par\noindent
  (c) $x=3$: 
    \par
    (c1)  $p=3m-1, d=3(np+3m^2+l), \nu=3(n+m)+1,  z=3l$
\par
    (c2)  $p=3m, d=3(np+m(3m+2)+l+1), \nu=3(n+m)+2,  z=3l+2$ 
\par
    (c3)  $p=3m+1, d=3(np+(3m+1)(m+1)+l+1), \nu=3(n+m+1),  z=3l+2$
 \bigskip\par \noindent
  where $m\in {\Bbb{Z}}^{+}$, $p\geq 4$, $n\in {\Bbb{Z}}^{\geq 0}$,  $l=0$ if $m=1$ and $l=0,1,2,...m-2$ if $m \geq  2$.

 \end{lem}
  \begin{proof}
  The initial trees have the form given in Fig. 4 (left), which correspond to ${\mathcal{B}}^{(3)}_{d_{0},\nu_{0}, \epsilon_{0}}(w)$ with 
  $d_{0}=x+p+(p-1)y+z,\nu_{0}=y,\epsilon_{0}=z$. We look for the values of $d_{0}=3q$ and $\nu_{0}$ satisfying $\lfloor\frac{d_{0}}{\nu_{0}+1}\rfloor=\lfloor\frac{d_{0}-1}{\nu_{0}}\rfloor-1=p-1$ (see Section 3). We find the following solutions for $x$ and $z$ ($p\geq 4, m\in {\Bbb{Z}}^{+}, l=0$ if $m \leq 2$ and $l=0,1,2,...m-2$ if $m > 2$)
 \bigskip\par \noindent  
    (a) $x=1$: 
  \par
  (a1) $p=3m-1, q=m(3m-2)+l+1, \nu_{0}=3m-1, z=3l+1$
  \par
    (a2) $p=3m, q=3m^2+l+1, \nu_{0}=3m,  z=3l+2$ 
  \par
   (a3)  $p=3m+1, q=m(3m+2)+l+1, \nu_{0}=3m+1,  z=3l+1$ 
\par\noindent
  (b) $x=2$:
  \par
  (b1) $p=3m-1, q=m(3m-1)+l+1, \nu_{0}=3m,  z=3l+2$ 
\par
 (b2)  $p=3m, q=m(3m+1)+l+1, \nu_{0}=3m+1, z=3l+2$  
\par
 (b3) $p=3m+1, q=3m(m+1)+l+1, \nu_{0}=3m+2, z=3l$ 
\par\noindent
  (c) $x=3$: 
    \par
    (c1)  $p=3m-1, q=3m^2+l, \nu_{0}=3m+1, z=3l$
\par
    (c2)  $p=3m, q=m(3m+2)+l+1, \nu_{0}=3m+2, z=3l+2$ 
\par
    (c3)  $p=3m+1, q=(3m+1)(m+1)+l+1, \nu_{0}=3(m+1), z=3l+2$

 \bigskip\par 
 We now add $3n$ edges as in Lemma 2.2 and we get the series for ${\mathcal{B}}^{(3)}_{d,\nu, \epsilon}(w)$ with $d=d_{0}+3np, \nu=\nu_{0}+3n, \epsilon=\epsilon_{0}, n\in {\Bbb{Z}}^{\geq 0}$.  They satisfy $ \lfloor \frac{d}{\nu+1}\rfloor=p-1,  \lfloor \frac{d-1}{\nu}\rfloor=p$ (see Eq. (3.9)). In  Fig. 5(a),(b) we show two examples which correspond to $(x,z,y)=(1,1,4), (3,2,8)$ respectively: $d=12n+18, \nu=3n+4, n\in {\Bbb{Z}}^{\geq 0}, \epsilon=1$ (Fig. 5(a)) and  $d=18n+51, \nu=3n+8, n\in {\Bbb{Z}}^{\geq 0}, \epsilon=2$
(Fig. 5(b), where we have used the notation of Fig. 4).
   \end{proof}
   
      \begin{figure}[h]
 \includegraphics[width=30pc]{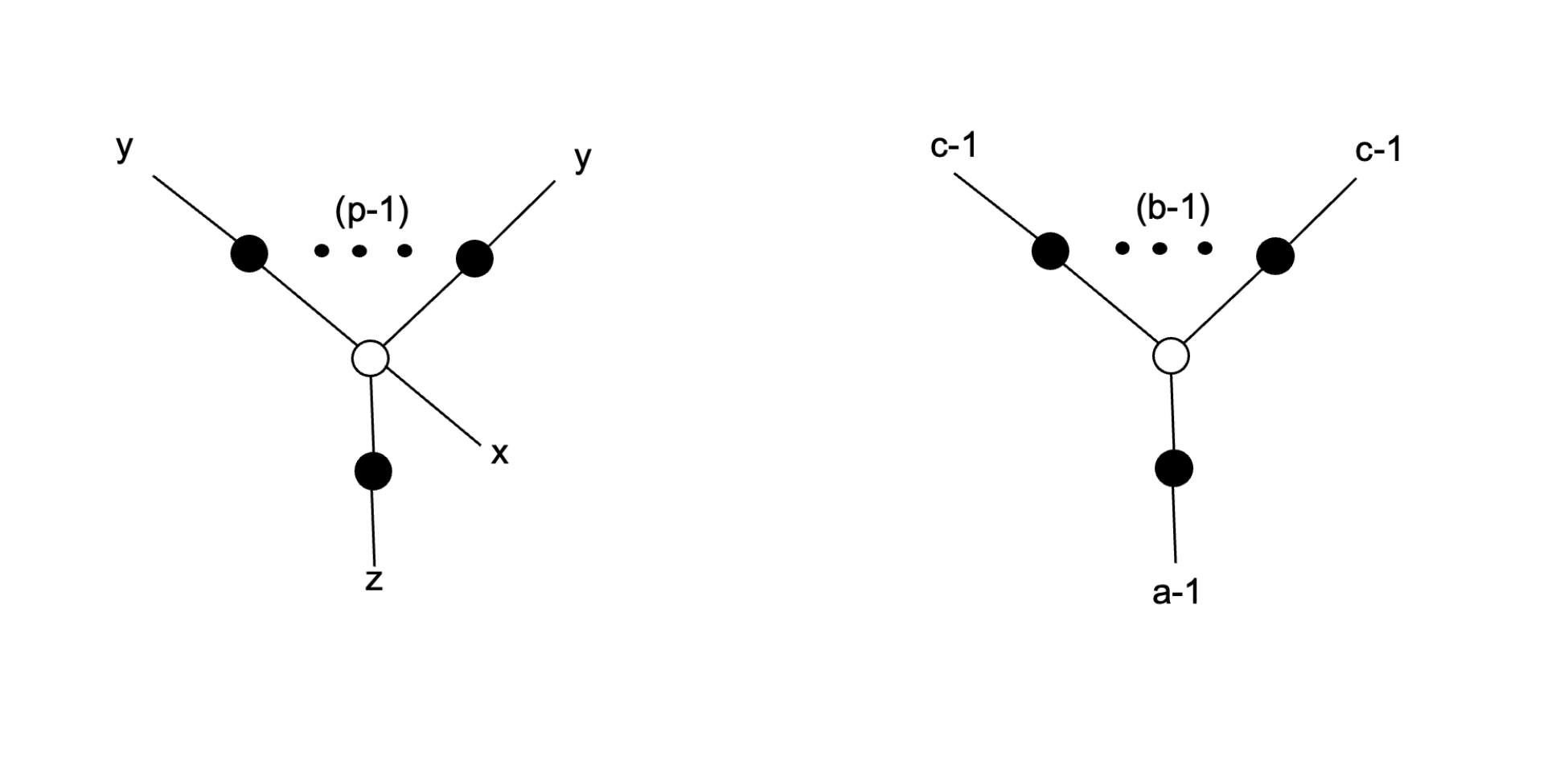}
\caption{\label{label}Plane trees for (left) Belyi polynomials in Lemma 2.3 with $n=0$; (right) Belyi polynomials ${\mathcal{B}}_{a +c(b-1),c-1, a-1}(w)$ with explicit equations given by $G_{a,b,c}(w)$. The number of edges with label $y$ or $c-1$ is indicated in brackets above the suspension points.}
\end{figure}

   \begin{figure}[h]
 \includegraphics[width=30pc]{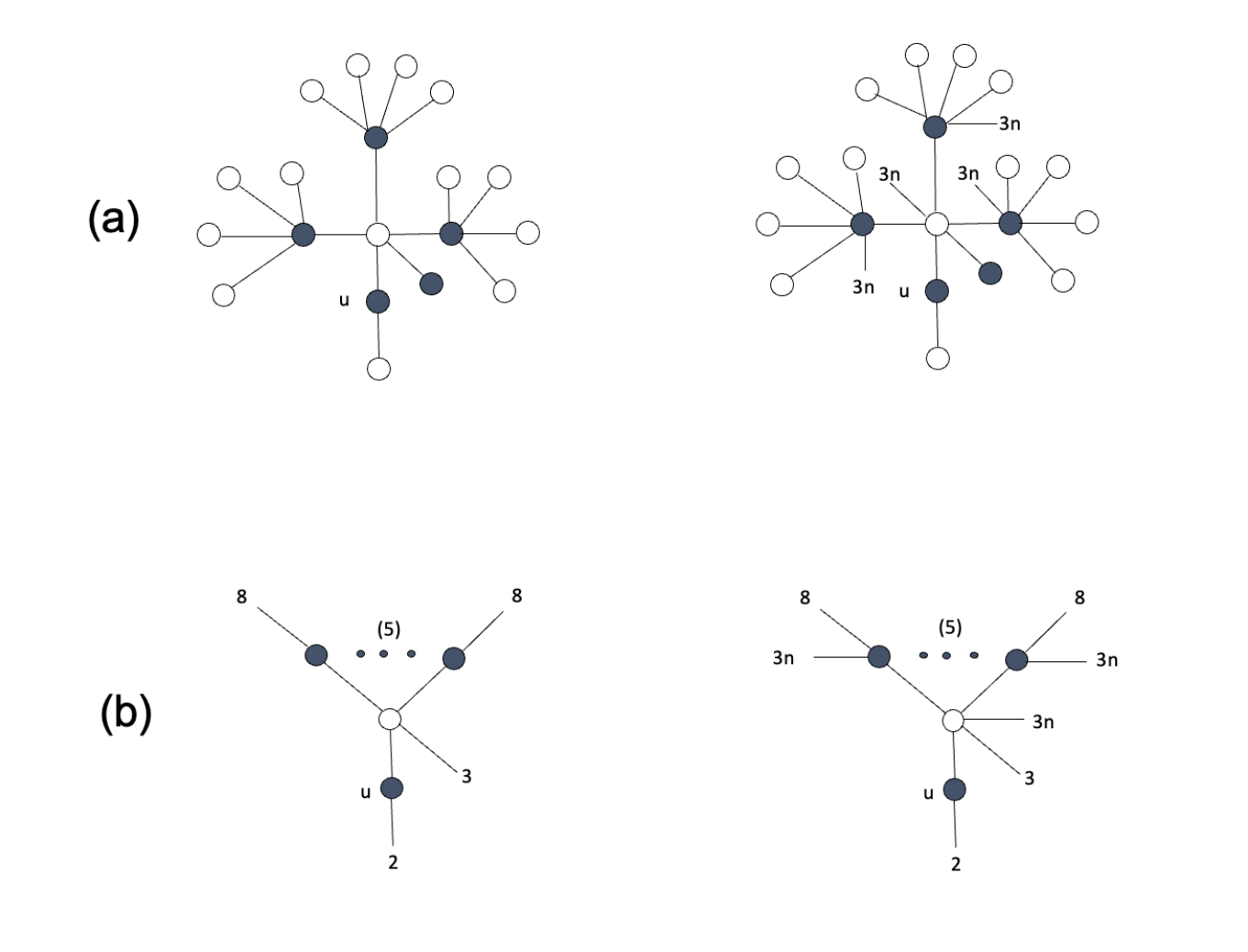}
\caption{\label{label}Plane trees for ${\mathcal{B}}^{(3)}_{d,\nu, \epsilon}(w)$: (a) $d=12n+18, \nu=3n+4, \epsilon=1$, $n=0$  (left)  and $n>0$ (right), (b) $d=18n+51,  \nu=3n+8, \epsilon=0$, $n=0$  (left)  and $n>0$ (right)}
\end{figure}

  The analysis of the associated trees shows that in certain cases some polynomials given in Lemma 2.3 coincide with some given in Lemmas 2.1 and 2.2. We now use the notation ${\mathcal{B}}^{(1)}[n,m]$, ${\mathcal{B}}^{(2)}[j,n,m,l]$ and ${\mathcal{B}}^{(3)}[x,p,n,m,l]$  to compare the polynomials. For Lemma $2.3$ $(x=2)$ and Lemma $2.1$ we have ${\mathcal{B}}^{(3)}[2, 3m+1,n,m,0]={\mathcal{B}}^{(1)}[n+1,m]$ with $n\in {\Bbb{Z}}^{\geq 0}$ and $m\in {\Bbb{Z}}^{+}$. As an example we can see that the tree in Fig.2(a) (right) is the same as the tree of the case $(x,z,y)=(2,0,5)$ in Lemma 2.3 with $d=12n+21,\nu=3n+5, \epsilon=0, n=0$, hence ${\mathcal{B}}^{(3)}[2, 4,0,1,0]={\mathcal{B}}^{(1)}[1,1]$. The cases with $l \neq 0$ in Lemma 2.3 are not included in Lemma 2.1 and the initial trees, associated to ${\mathcal{B}}^{(1)}[0,m]$ in Lemma 2.1, are not included in Lemma 2.3. The comparison between Lemma $2.3$ $(x=3)$ and Lemma $2.2$ gives ${\mathcal{B}}^{(3)}[3, 3m-1,n,m,l]={\mathcal{B}}^{(2)}[1,n+1,l,m-1]$ with $n\in {\Bbb{Z}}^{\geq 0}$, $m\in {\Bbb{Z}}^{\geq 2}$, $l \geq m+1$ and ${\mathcal{B}}^{(3)}[3, 3m+1,n,m,l]={\mathcal{B}}^{(2)}[0,n+1,l+1,m]$ for $n\in {\Bbb{Z}}^{\geq 1}$,  $m\in {\Bbb{Z}}^{\geq 1}$, $l \geq m$ if $m \neq 2$, $l \geq 3$ if $m = 2$. In addition to the initial trees, that correspond to ${\mathcal{B}}^{(2)}[0,0,m,l], {\mathcal{B}}^{(2)}[1,0,m,l]$, there are other cases in Lemma 2.2 which are not included in Lemma 2.3, like ${\mathcal{B}}^{(2)}[1,n,m,m]$ with $n\in {\Bbb{Z}}^{\geq 1}$, $m\in {\Bbb{Z}}^{\geq 0}$ and ${\mathcal{B}}^{(2)}[0,n,m,m]$ for $n\in {\Bbb{Z}}^{\geq 1}, m\in {\Bbb{Z}}^{\geq 2}$.

   \section{Surfaces with many $A$ singularities}
   The affine equations of the surfaces considered in this paper are obtained with the polynomials  ${\mathcal{B}}^{(t)}$ discussed in Section 2 and polynomials related to the family 
    \begin{equation} 
\hat{J}_{d,\tau}(x,y):=\lambda_{d,\tau}  \prod_{\mu} L_{d,\tau,\mu}\left(x,y\right)
   \end{equation} 
   where
$$L_{d,\tau,\mu}(x,y):=y+\left({\rm cos}\left(\frac{2\pi}{d}\left(\frac{6\mu-1}{6}-\frac{\tau}{\pi}\right)\right)-x\right){\rm tan}\left(\frac{\pi}{d}\left(\frac{6\mu-1}{6}-\frac{\tau}{\pi}\right)\right)+{\rm sin}\left(\frac{2\pi}{d}\left(\frac{6\mu-1}{6}-\frac{\tau}{\pi}\right)\right)$$ 
with $ \mu=- \lfloor \frac{d-2}{2} \rfloor,- \lfloor \frac{d-2}{2} \rfloor+1,- \lfloor \frac{d-2}{2} \rfloor+2,..., \lfloor \frac{d+1}{2} \rfloor, (x,y)\in {\Bbb{R}}^2$ and $\tau \in \Bbb{R}$.
  The parameters in Eq. (3.1) are $\lambda_{d,(6m-3d-1)\frac{\pi}{6}} =(-1)^{m}2 d$ and  $\lambda_{d,\tau}  = 2{\rm cos}(\tau+\frac{d\pi}{2}+\frac{2\pi}{3})$ if $\tau \neq (6m-3d-1)\frac{\pi}{6}$, $m \in \Bbb{Z}$ (the line $L_{d,\tau,\mu}(x,y)=0$ parallel to the $y$-axis is interpreted as the line $x+1=0$).  
  \par
The following result will be necessary in what follows (\cite{esc16}, Lemma 1)

    \begin{lem}     The real polynomial $\hat{J}_{d,0}(x,y)$ has ${d \choose 2}$ critical points with critical value $0$. The number of points with critical value $8$ is $\frac{d(d-3)}{6}$ if $d=0$ mod 3, and $\frac{(d-1)(d-2)}{6}$ otherwise. The number of critical points with critical value $-1$ is $\frac{d^{2}}{3}-d+1$ for $d=0$ mod 3, and $\frac{(d-1)(d-2)}{3}$ otherwise. 
     \end{lem}

  \par
  In (\cite{esc18}, Lemma 4.2) we have shown that
     \begin{equation} 
      {\mathcal{J}}_{d}(x,y):= \hat{J}_{d,0}\left(x,\frac{y}{ \sqrt{3}}\right)
           \end{equation}      
are defined over ${\Bbb{Q}}$, and also that we can get the Chebyshev polynomials from 
     \begin{equation} 
      {\mathcal{J}}_{d}(z,0)=-2   T_{d}\left(\frac{z-1}{2}\right)+1
            \end{equation}  
     \par
There are families of degree $d$ nodal hypersurfaces, defined over ${\Bbb{Q}}$, with affine equations 
       \begin{equation} 
       {\mathcal{J}}_{d}(x,y)+\frac{1}{4}(3-{\mathcal{J}}_{d}(2z+1,0))=0 
          \end{equation} 
     with  $(x,y,z) \in {\Bbb{R}^3}$   and
           \begin{equation} 
       {\mathcal{J}}_{d}(x,y)-{\mathcal{J}}_{d}(z,w)=0
                 \end{equation}  
     with  $(x,y,z,w) \in {\Bbb{R}^4}$      
having many singularities as shown in (\cite{esc18}, Theorem 4.3).  
 \par
Nodal hypersurfaces obtained by generalising Eqs. (3.4), (3.5) have been studied from the point of view of their invariants and projective rigidity in \cite{esc24}. Related to ${\mathcal{J}}_{d}$ are also certain maximising curves \cite{esc25}. The proof of Theorem 4.3 in \cite{esc18} is based on the characterisation of the critical points of ${\mathcal{J}}_{d}$. This analysis, given in Lemma 3.1, together with the results in Section 2 leads to 

     \begin{prop} 
The degree $d=3q$ surfaces with affine equations  
           \begin{equation} 
{\mathcal{J}}_{d}(u,v)+\frac{{\mathcal{B}}^{(t)}_{d,\nu,\epsilon}(w)+1}{2}=0,
      \end{equation} 
 where $(u,v,w) \in {\Bbb{C}^3}$ and $s=k,b,p$ for $t=1,2,3$ respectively, have 
           \begin{equation} 
\frac{d(d-1)(s-1)}{2}+\frac{d(d-3)}{3}+1
           \end{equation} 
singularities of type $A_{\nu}$.
 \end{prop}
 \begin{proof}
 We denote by $N_{\zeta}(\mathcal{J})$ the number of critical points of ${\mathcal{J}}_{d}(u,v)$ with critical value $\zeta$ and by $N_{\zeta}(\mathcal{B}, \rho)$ the number of critical points of ${\mathcal{B}}^{(t)}_{d,\nu,\epsilon}(w)$ with critical value $\zeta$ and multiplicity $\rho$. Then the number of singularities of type $A_{\nu}$ of the surfaces described by Eq. (3.6) is $N_{0}(\mathcal{J})N_{-1}(\mathcal{B}, \nu)+N_{-1}(\mathcal{J})N_{1}(\mathcal{B}, \nu)$.
 \par
According to Lemma 3.1, for $d=3q+\alpha, \alpha \in\{0,1,2\}, p \in {\Bbb{Z}}^{+}$, the polynomial ${\mathcal{J}}_{d}(u,v)$ has $N_{0}(\mathcal{J})={d \choose 2}$ critical points with critical value $\zeta=0$, $N_{8}(\mathcal{J})=(d^{2}-3d+2\lceil\frac{\alpha}{2}\rceil)/6$ critical points with $\zeta=8$, and $N_{-1}(\mathcal{J})=(d^{2}-\lceil\frac{\alpha}{2}\rceil)/3-d+1$ critical points with $\zeta=-1$ \cite{esc18}. 
 \par
The derivatives of ${\mathcal{B}}^{(t)}_{d,\nu,\epsilon}(w)$ ($s=k,b,p$ for $t=1,2,3$ respectively) are
 \begin{equation}
\frac{d {\mathcal{B}}^{(t)}_{d,\nu,\epsilon}(w)}{dw}=(w-w_{0})^{\nu}(w-w_{u})^{\epsilon}\prod_{l=1}^{s-1}(w-w_{l})^{\nu}
\end{equation}
with ${\mathcal{B}}^{(t)}_{d,\nu,\epsilon}(w_{l})=-1, l=1,2,...,s-1; {\mathcal{B}}^{(t)}_{d,\nu,\epsilon}(w_{0})=1$ and  ${\mathcal{B}}^{(t)}_{d,\nu,\epsilon}(w_{u})=-1$ if $\epsilon \neq 0$. Considering the results of Lemmas 2.1, 2.2 and 2.3  we have $N_{-1}(\mathcal{B}, \nu)=s-1$ for $s=k,b,p$ and  $N_{1}(\mathcal{B}, \nu)=1$. On the other hand, the number of critical points of ${\mathcal{J}}_{d}$ for $d=3q$ is $N_{0}(\mathcal{J})=\frac{d(d-1)}{2}, N_{8}(\mathcal{J})=\frac{d(d-3)}{6}, N_{-1}(\mathcal{J})=\frac{d(d-3)}{3}+1$, hence, we see that the number of type $A_{\nu}$ singularities of the surfaces with affine equations ${\mathcal{J}}_{d}(u,v)+({\mathcal{B}}^{(t)}_{d,\nu,\epsilon}(w)+1)/2=0$ is given by Eq. (3.7). In addition, when $\epsilon \neq 0$, they also have $\frac{d(d-1)}{2}$ singularities of type $A_{\epsilon }$, because $N_{-1}(\mathcal{B}, \epsilon)=1$ and  $N_{1}(\mathcal{B}, \epsilon)=0$. 
   \end{proof}
   
 The lower bounds for the maximal number of $A_\nu$ singularities given in \cite{lab06} are improved slightly. In that paper it is shown that there exists a polynomial $M^{j}_{d}(w)$ of degree $d$ with $\lfloor \frac{d}{j+1}\rfloor$ critical points of multiplicity $j$ with critical value $-1$ and $\lfloor \frac{d-1}{j}\rfloor-\lfloor \frac{d}{j+1}\rfloor$ critical points with critical value $+1$. If $F_{d}^{\bf{A}_{2}}(u,v)$ denotes the folding polynomial associated to the root lattice $\bf{A}_{2}$ obtained in \cite{wit88, hof88}, then the surfaces with affine equations $F_{d}^{\bf{A}_{2}}(u,v)+M^{j}_{d}(w)=0$ have $\frac{d(d-1)}{2}\lfloor \frac{d}{j+1}\rfloor+\frac{d(d-3)}{3}( \lfloor \frac{d-1}{j}\rfloor-\lfloor \frac{d}{j+1}\rfloor)$ singularities of type  $A_{j}$ when $d \equiv 0$ (mod 3).
\par 
The polynomials ${\mathcal{B}}^{(t)}_{d,\nu,\epsilon}(w), t=1,2,3$ have degrees $d$ and multiplicities $\nu$ satisfying 
         \begin{equation} 
 \lfloor \frac{d}{\nu+1}\rfloor=s-1,  \lfloor \frac{d-1}{\nu}\rfloor=s
\end{equation} 
for $s=k,b,p$, hence the surfaces ${\mathcal{J}}_{d}(u,v)+({\mathcal{B}}^{(t)}_{d,\nu,\epsilon}(w)+1)/2$  have one more $A_\nu$ singularity than $F_{d}^{\bf{A}_{2}}(u,v)+M^{\nu}_{d}(w)$.

    \section{Some explicit expressions in terms of classical Jacobi polynomials}
    In general Belyi polynomials can be computed only for low degree by using Groebner basis (see \cite{esc14a, esc14b} for some examples by using Singular computing tool \cite{gre08}). In this section we analyse some cases that can be obtained in terms of classical Jacobi polynomials $P_{k}^{l,m}(z)$ \cite{abr70}, \cite{sze75}. We extend the results in \cite{esc14a} by considering the three-parameter family of degree $d=a+(b-1)c$ polynomials 
   \begin{equation}    
     G_{a,b,c}(w)=z^{a}P_{b-1}^{a/c,-b}(1-2w)^{c}
   \end{equation} 
   One of the zeros of $\frac{dG_{a,b,c}}{dw}$ with critical value $\zeta=0$ has order $\nu=a-1$, there are also $b-1$ zeros with critical value $\zeta=0$ and $\nu=c-1$, while the unique remaining root $w=1$ has $\nu=b-1$ with critical value $\zeta=1$. The general form of the trees for $G$ can be seen in Fig. 4 (right).  The polynomials used in this work have $c=b$. We now show that some polynomials in Lemmas 2.1 and 2.2 can be described in terms of $G_{a,b,c}(w)$.  No explicit expressions have been found for ${\mathcal{B}}^{(3)}_{d,\nu, \epsilon}(w)$ in terms of classical Jacobi polynomials (see Fig. 4).
 \bigskip\par    
 For the polynomials in Lemma 2.1 corresponding to $n=0$ in Eq.(2.1) we have that
   \begin{equation} 
   {\mathcal{B}}^{(1)}_{9m^2,3m-1}(w)=G_{3m,3m,3m}(w)=w^{3m}P_{3m-1}^{1,-3m}(1-2w)^{3m}
     \end{equation}  
        has $w=0$ with $\zeta=0, (\nu=3m-1)$, $w=1$ with $\zeta=1, (\nu=3m-1)$ and $3m-1$ points with $\zeta=0, (\nu=3m-1)$.      
 \bigskip\par\noindent 
   {\it Example 4.1.} $G_{3,3,3}(w)$  was already considered in \cite{esc14a}, Prop. 2, for constructing a nonic surface with 127 cusps, but the  surface is not defined over the rationals, by contrast with ${\mathcal{J}}_{9}(u,v)+({\mathcal{B}}^{(1)}_{9,2}(w)+1)/2$. The polynomial 
   $${\mathcal{B}}^{(1)}_{36,5}(w)=G_{6,6,6}(w)=w^6 (6 - 15 w + 20 w^2 - 15 w^3 + 6 w^4 - w^5)^6$$ 
   has critical points $w=0, (\nu=5), w=2, (\nu=5), w=\frac{1 \pm i {\sqrt{3}}} {2},  (\nu=5), w=\frac{3 \pm i {\sqrt{3}}}{2}, (\nu=5),$ with critical value $\zeta=0$  and $w=1, (\nu=5)$ with $\zeta=1$. The surface ${\mathcal{J}}_{36}(u,v)+({\mathcal{B}}^{(1)}_{36,5}(w)+1)/2$ is also defined over ${\Bbb{Q}}$ and has 4177 complex singularities $A_{5}$ .
 \bigskip\par 
The polynomials corresponding to the initial trees in Lemma 2.2 can be expressed as 
   \begin{equation} 
 {\mathcal{B}}^{(2)}_{a+b(b-1),b-1,a-1}(w)= G_{a,b,b}(w)= w^{a}P_{b-1}^{a/b,-b}(1-2w)^{b}
     \end{equation}  
   for $a=1, 3, 4, 6, 7, 9, 10, ..., 3m, 3m+1,...$ and $b>a$.  In particular:
  \par   
  1) $ {\mathcal{B}}^{(2)}_{3m+3l(3l-1),3l-1,3m-1}(w)= G_{3m,3l+1,3l+1}(w)$, $m\in {\Bbb{Z}}^{+}, l=m, m+1, m+2,...$.     \par
  2) $ {\mathcal{B}}^{(2)}_{3m+1+(3l+2)(3l+1),3l+1,3m}(w)= G_{3m+1,3l+2,3l+2}(w)$, $m\in {\Bbb{Z}}^{\geq 0}, l=m, m+1, m+2,...$. 
    
   \bigskip\par\noindent 
   {\it Example 4.2.} 
     \par
     1)  The lower bound $\mu_{A_{3}}(15)\ge 376$ was obtained in \cite{esc14a}. An explicit expression of ${\mathcal{B}}^{(2)}_{15,3,2}(w)= G_{3,4,4}(w)$ is given in the proof of Prop. 4. The surface is defined over ${\Bbb{C}}$, whereas ${\mathcal{J}}_{15}(u,v)+({\mathcal{B}}^{(2)}_{15,3,2}(w)+1)/2$ is defined over ${\Bbb{Q}}$, and both have 376  $A_{3}$ and 105 $A_{2}$ singularities.
     
          \begin{figure}[h]
 \includegraphics[width=15pc]{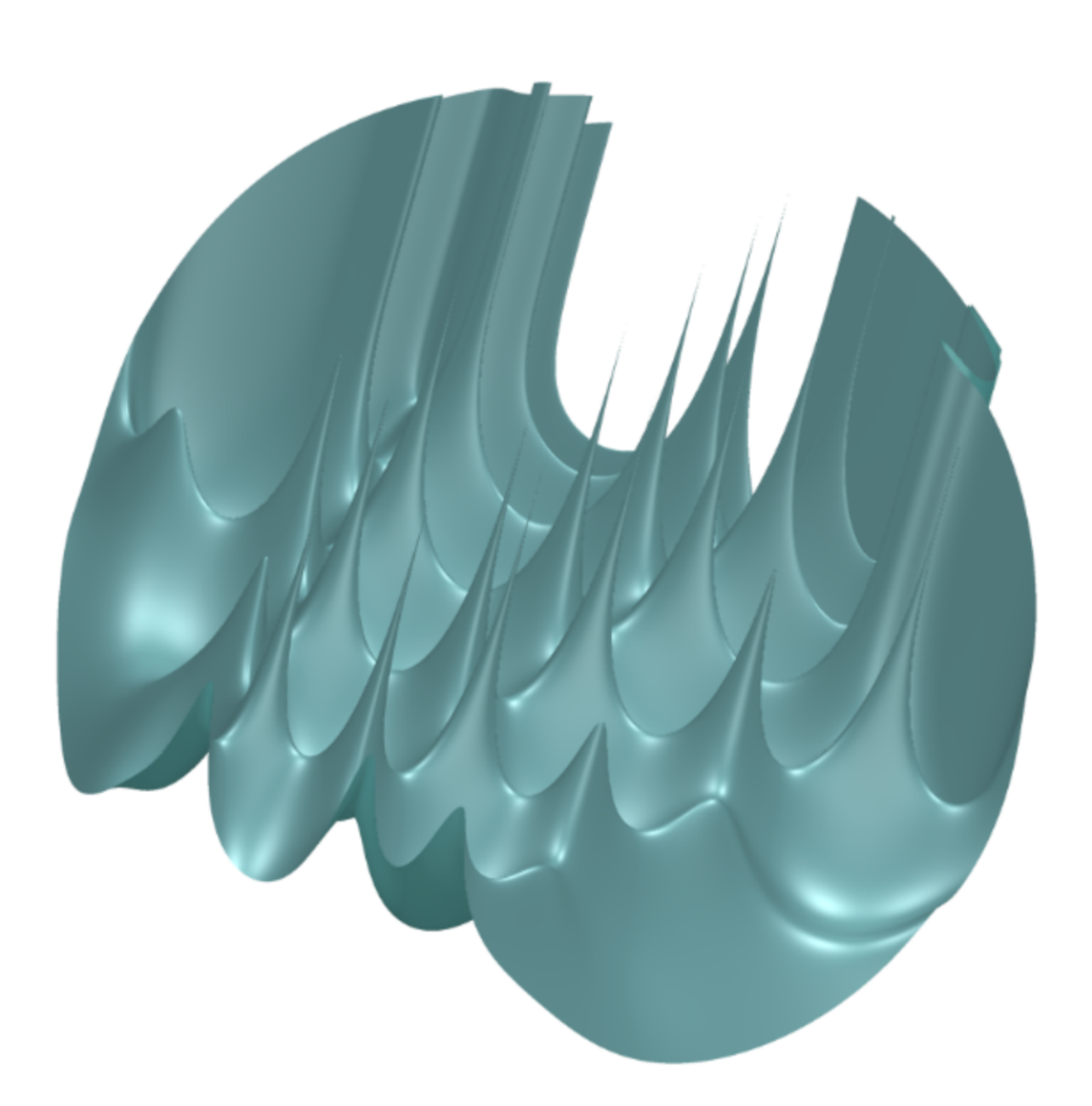}
\caption{\label{label} A fragment of the real part of a degree 21 surface with $A_{4}$ singularities.}
\end{figure}

     \par
 2)   ${\mathcal{B}}^{(2)}_{3,1,0}(w)=G_{1,2,2}(w)$ is associated with the starting tree of the series with $d=6n+3, \nu=3n+1$, and the associated surface has the same number of nodes as the Cayley cubic surface. This case is studied  in Prop. 3.6 in \cite{esc14b} where we derived the lower bound given by Eq (1.3). Explicit expressions for 
 ${\mathcal{B}}^{(2)}_{d,1,0}(w), d=9,15$ were presented in \cite{esc14a, esc14b} by using Groebner basis, and the surfaces obtained give 
 $\mu^{(\Bbb{R})}_{A_{4}}(9)\ge 55$  and  $\mu^{(\Bbb{R})}_{A_{7}}(15)\ge 166$. Another particular case is  $${\mathcal{B}}^{(2)}_{21,4,0}(w)= G_{1,5,5}(w)=\frac{w}{5^{20}} (924 - 616 w + 504 w^2 - 231 w^3 + 44 w^4)^5$$ 
which gives the degree 21 surface ${\mathcal{J}}_{21}(u,v)+({\mathcal{B}}^{(2)}_{21,4,0}(w)+1)/2$, with 757 complex singularities $A_{4}$. A representation of a fragment of this surface real part with the {\it Surfer} visualisation tool can be seen in Fig. 6.

\par

\end{document}